\documentclass[11pt]{amsart}
\usepackage{amsbsy}
\usepackage{amsmath}
\usepackage{amsbsy}
\usepackage{graphicx}
\usepackage{enumerate}
\usepackage[cp1250]{inputenc}
\usepackage{indentfirst}
\usepackage{latexsym}
\usepackage{amssymb}
\usepackage{amsfonts}
\usepackage{color}
\usepackage{dsfont}
\usepackage{mathtools}
\usepackage{amsthm}

\usepackage{tikz}
\usetikzlibrary{arrows}
\usepackage{xcolor}
\usepackage{xypic}
\usepackage{cite}
\usepackage{float}
\usepackage{filecontents}
\usepackage{graphicx}
\usepackage{epstopdf}
\usepackage{epsfig}
\definecolor{darkgreen}{rgb}{0,.7,.3}

\theoremstyle{definition}
\newtheorem{definition}{Definition}[section]
\newtheorem{remark}[definition]{Remark}

\theoremstyle{plain}
\newtheorem{lemma}[definition]{Lemma}
\newtheorem{theorem}[definition] {Theorem}
\newtheorem{proposition}[definition] {Proposition}

\newtheorem{corollary}[definition]{Corollary}

\newcommand{\PP}{\text{\rm PP}}
\newcommand{\WP}{\text{\rm WP}}
\newcommand{\CP}{\text{\rm CP}}
\newcommand{\im}{\text{\rm im}}

\newcommand{\F}{\text{\rm F}}

\begin{document}

\title{Notes about decidability of exponential equations}

\begin{abstract}
We study relationship among versions of the Knapsack Problem
where variables take values in $\mathbb{Z}$ and the number of them is fixed.
\end{abstract}

\author{Oleg Bogopolski}
\address{{Sobolev Institute of Mathematics of Siberian Branch of Russian Academy
of Sciences, Novosibirsk, Russia}\newline
{and D\"{u}sseldorf University, Germany}}
\email{Oleg$\_$Bogopolski@yahoo.com}

\author{Aleksander Ivanov (Iwanow)}
\address{Department of Applied Mathematics, Silesian Univesity of Technology,
ul. Kaszubska 23, 44 - 101 Gliwice, Poland}
\email{Aleksander.Iwanow@polsl.pl}


\maketitle

\section{Introduction}

%Let us assume that $G$ is a recursively  presented group and let $X$ be a set of generators.
%Below each element of $G$ is viewed as a word in $X$.
%When $w\in G$ we denote by $|w|_X$ the $X$-length of the shortest form $w$ in $G$.
%We denote by $\F (X)$ the group freely generated by $X$.
%If $w\in \F (X)$ is in the reduced form, then $|w|_X$ is just the $X$-length of $w$.
%In this case we often write $|w|$.  \parskip0pt
%By $\WP(G)$ and $\CP(G)$ we denote the word and the conjugacy problem in $G$, respectively.

Let $G$ be a group given by a recursive presentation $\langle X\,|\ R\rangle$.
For $w\in G$ we denote by $|w|_X$ the length of a shortest word in the alphabet $X\cup X^{-1}$ representing $w$.
The free group generated by $X$ is denoted by  $\F (X)$. The length of $u\in \F(X)$ with respect to $X$ will be often written as $|u|$.
By $\WP(G)$ and $\CP(G)$ we denote the word and the conjugacy problems for $G$, respectively. In
formulations of decision problems, we assume that elements of $G$ are represented by words in the alphabet $X\cup X^{-1}$.

Let $\PP[n]$ be the set of all tuples $\bar{g}=(g_0 ,g_1 ,\ldots ,g_n )$ from $G^{n+1}$
such that the equation
$g_0 = g^{z_1}_{1} \cdot \ldots \cdot g^{z_n}_n$ has a solution which is an $n$-tuple of integers.
%which is a tuple of integers.
We define the {\it power problem of order} $n$ for $G$ to be the membership problem for the set $\PP[n]$.
For brevity, we denote this decision problem again by $\PP[n]$.

In the case $n=1$ the decision problem $\PP [1]$ is known as the {\it power problem}, see \cite{MC} and \cite{OS2}.
Our notes are motivated by the following question.

\medskip

\begin{center}
{\em What is the relationship among the decision problems $\PP [n]$ for different $n$?}
\end{center}

\medskip

It is easy to see that decidability of the power problem implies decidability of $\WP$.
Furthermore, decidability of $\PP [n+1]$ implies decidability of
$\PP [n]$.
Indeed, in order to verify if $(g_0 ,g_1 ,\ldots ,g_n ) \in \PP [n]$
we just check if $(g_0 ,g_1 ,\ldots ,g_n ,e)\in \PP [n+1]$, where $e$ is the unit of $G$.
The main result of this paper shows that decidability of $\PP [1]$
does not imply decidability of $\PP [2]$.
\medskip

{\bf Theorem A.}
{\it There exists a finitely presented group with decidable $\PP[1]$ and undecidable $\PP[2]$.
Moreover, this group has decidable conjugacy problem.}
\medskip

The next note concerns estimation of possible solutions of
exponential equations arising in $\PP [n]$ by recursive functions on sizes of coefficients.
The motivation comes from the fact that this is a usual way to solve such equations.
In Proposition~\ref{unsolvable} we show that primitively recursive functions are not sufficient for this aim.
More information about the complexity of estimating functions for some interesting classes of groups can be found
in Remarks~\ref{complexity_remark_1} and~\ref{complexity_remark_2}.
%bounds this way might be quite complicated, see Section 4 and Proposition~\ref{unsolvable} there.
%In Proposition ... we show that there is a finitely presented group with decidable $\PP[1]$,
%for which there does not exist a {\it primitively recursive} function $f(x)$ where the size of a ``minimal''
%solution for the exponential equation of kind $u=v^z$ (if some solution exists) cannot be estimated by any

\medskip

In Section 5 we introduce decision problems $\PP[g,G^n ]$ and $\PP[G,\overline{g}]$ which can be considered as fragments of $\PP[n]$ for $G$.
We show that these fragments can have diverse r.e. Turing degrees in the same finitely presented group.
From a quite general Theorem~\ref{1/2}, we deduce the following statement.

\medskip

{\bf Theorem B.} (see Corollary~\ref{cor1/2})
%\begin{corollary} \label{cor1/2}
{\it There exists a finitely presented torsion-free group $G$ with decidable conjugacy problem and undecidable $\PP[1]$ such that any\break r.e. Turing degree is realised as the Turing degree of the problem
$\PP[g,G]$ for appropriate $g\in G$.}
%\end{corollary}

\medskip

The methods which we use strongly depend on papers
\cite{O1} -- \cite{OS2}.
%%%%%
%To provide exact references, we included in Section 3 a short %description of main steps in the embedding construction of %Ol'shanskii and Sapir following~\cite{OS2}.

\medskip

In the places where arguments are of computability theory flavor we follow the terminology of \cite{soare} (in particular we write ``computable" instead of ``recursive").
In the remaining parts of the paper we keep the traditional terminology.

\begin{remark}
The Knapsack Problem of $G$ is the decidability problem of
recognizing of all tuples $\bar{g} \in G^{n+1}$, $n\in \mathbb{N}$,
such that $\bar{g} = (g_0 ,g_1 ,\ldots g_n )$
and the equation
$g_0 = g^{z_1}_{1} \cdot \ldots \cdot g^{z_n}_n$ has a solution which
a tuple of {\em natural numbers}.
It is easy to see that decidability of the Knapsack Problem in $G$ implies decidability of the set $\bigcup_{n\in \mathbb{N}} \PP [n]$.
The Knapsack Problem  was introduced in \cite{MNU}.
It has become a very active area of research
where the interplay between group theoretic properties and algorithmic complexity is a typical topic,
see \cite{Dudkin}, \cite{Frenkel},  \cite{GKLZ}, \cite{KLZ},
\cite{Lohrey_1}, \cite{Lohrey_2}, \cite{LZ_1}, \cite{LZ_2}, and~\cite{MT}.
On the other hand questions similar to ones studied in our paper seem to be open for the Knapsack Problem.
We will comment this below.
\end{remark}

\begin{remark}
Several authors consider exponential equations in the following form:
\[
h_1 g_1^{z_1} h_2 g^{z_2}_2 \dots h_n g_n^{z_n}=1.
\]
It is worth noting that this equation can be rewritten as
\[
f_1^{z_1}f_2^{z_2}\dots f_n^{z_n} = f_0,
\]
where $f_0=(h_1 \dots h_n)^{-1}$ and $f_i=(h_1\dots h_i)g_i(h_1 \dots h_i)^{-1}$ for $i=1,\dots,n$.\break
In particular, the question about decidability of $\PP [n]$ does not depend on the form of exponential equations.
\end{remark}

\section{A recursively presented group with decidable PP[1] and undecidable PP[2]}

The main purpose of this section is the following weaker version of Theorem A.

\begin{proposition}\label{four}
There exists a recursively presented group $G$ such that
$\PP[1]$ is  decidable, but $\PP[2]$ is undecidable.
\end{proposition}

In the proof of this proposition we use
the following lemmas. The first one is obvious.

\begin{lemma}\label{one}
Let $w$ and $u$ be two nontrivial elements of the free group $\F(X)$.
If $w=u^z$ for some $z\in \mathbb{Z}$, then
$|z|\leqslant |w|_X$.
\end{lemma}

\begin{lemma}\label{two} Let $w(a,b,c)$ be a nonempty reduced cyclic word in $\F(a,b,c)$ and let
%containing at least one $c$ or $c^{-1}$.Let
$$
M=\max \{|z|: w\hspace*{2mm} {\text {\rm has a subword of the form}}\hspace*{2mm} a^z
{\text{\rm or}}\hspace*{2mm} b^z\}.
$$
Suppose that $m>M$.
Then $w(a,b,a^mb^m)\neq 1$ in $\F(a,b)$. Moreover, if
$$
w(a,b,a^mb^m)=v(a,b)^z
$$
for some $v(a,b)\in \F(a,b)$ and $z\in \mathbb{Z}$, then $|z|\leqslant |w(a,b,c)|$.
\end{lemma}

\medskip

{\it Proof.} We assume that $w$ contains at
least one $c$ or $c^{-1}$ (otherwise the statement is obvious).
For any word $u(a,b,c)\in \F(a,b,c)$ let $T(u)\in \F(a,b)$ be the reduced word obtained from $u$ by substitution $c \rightarrow a^mb^m$ followed by reduction. If $u(a,b,c)$ is considered as a cyclic word, we
denote by $CT(u)$ the cyclic word, obtained from $u$ by substitution $c \rightarrow a^mb^m$ followed by cyclic reduction.
The following claim can be proved by induction on the number of occurrences of letters $c^{\pm}$ in $u$.

{\it Claim.} Let $u$ be a subword of the cyclic word $w$ from lemma.

1) If $u$ begins with $c$ (resp. with $c^{-1}$) then $T(u)$ begins with $a^m$
(resp. with~$b^{-m}$).

2) If $u$ ends with $c$ (resp. with $c^{-1}$), then $T(u)$ ends with $b^m$
(resp. with $a^{-m}$).

\medskip

Now we show that $CT(w)$ contains at least two syllables.
In particular, this will imply
$w(a,b,a^mb^m)\neq 1$ in $\F(a,b)$.
The case where $w$ contains only one occurrence of $c$ or $c^{-1}$ is obvious. We assume that $w$ contains at
least 2 occurrences of $c^{\pm 1}$.

The case, where the occurrences of $c$ and $c^{-1}$ in $w$ alternate,
i.e. if $w=u_1cu_2c^{-1}\dots cu_sc^{-1}$, where $u_i\in \F(a,b)$ for all $i$, is easy. Indeed, each subword $c^{-1}u_ic$ gives at least one $a$-syllable, and each subword $cu_ic^{-1}$ gives at least one $b$-syllable.

\medskip

Now suppose that $w$ does not have this form. Then the cyclic word $w$ has the form $cpcq$ or $c^{-1}pc^{-1}q$, where $p\in \F(a,b,c)$ and $q\in \F(a,b)$. Using inversion, we may assume the first case.
Note that $CT(cpcq)$ can be obtained from $T(cpc)q$ by cyclic reduction.
By the claim above, the word $T(cpc)$ begins with $a^m$ and ends with $b^m$. By assumption, the absolute values of exponents of all syllables in the word $q$ are less than $m$. Therefore, after cyclic reduction of $T(cpc)q$, some nontrivial parts of the syllables $a^m$ and $b^m$ remain.

Thus, we have proved that the cyclic word $CT(w)$ contains $n$ syllables for some $n\geqslant 2$. Hence the cyclic word $CT(v)$ also contains at least two syllables, say $m$. Then $z=n/m\leqslant n/2$. It remains to note that $n\leqslant 2|w|$. The latter is valid since, after substitution $c\rightarrow a^mb^m$ in $w$, the total number of $a$-syllables and $b$-syllables increases by at most $2k$, where $k$ is the number of occurrences of $c^{\pm 1}$ in $w$.\hfill $\Box$

\medskip

{\it Proof of Proposition~\ref{four}}.
Our construction resembles McCool's example from~\cite{MC}.
Let $f:\mathbb{N}\rightarrow \mathbb{N}$ be a one-to-one recursive function with non-recursive range.
Consider the following infinite presentation:
\[
G=\big\langle \underset{i\in \mathbb{N}}{\cup}\{ a_i,b_i,c_i\}\,\,|\, c_{f(i)}=a_{f(i)}^ib_{f(i)}^i\,\, (i\in \mathbb{N})\big\rangle.\eqno{(2.1)}
\]
Let  $X=\underset{i\in \mathbb{N}}\cup X_i$, where $X_i=\{a_i,b_i,c_i\}$.
Let $H_i$ be the subgroup of $G$ generated by~$X_i$.
Then
\[
G=\underset{j\in \mathbb{N}}{\ast} H_j \eqno{(2.2)}
\]
where each $H_j$ is free and
\[
{\rm rk}(H_j)=
\begin{cases}
2 & \hspace*{2mm} {\text{\rm if}}\hspace*{2mm} j\in \im f,\\
3 & \hspace*{2mm} {\text{\rm if}}\hspace*{2mm} j\notin \im f.
\end{cases}
\]
{\bf Claim 1.}
The word problem is decidable for the presentation (2.1).

\medskip

{\it Proof.} Using the normal form of an element of the free product (2.2), we reduce $\WP (G)$ to the following problem.
Given $j\in \mathbb{N}$ and given a reduced nonempty word $w(a_j,b_j,c_j)$,
decide whether the corresponding element of $H_j$ is trivial or not.
The difficulty is that we do not know whether $j\in \im (f)$ or not.

%We may assume that
%\begin{itemize}
%\item $w$ contains at least one $c_j^{\pm 1}$.
%\end{itemize}
%Indeed, otherwise $w$ lies in the free subgroup $F(a_j,b_j)$ of $H_j$ and therefore is nontrivial.

From now on we consider $w(a_j,b_j,c_j)$ as a nonempty reduced {\it cyclic} word in $\F(a_j,b_j,c_j)$.
Let $M$ be the maximum of absolute values of exponents of $a_j$ and $b_j$ in the word $w(a_j,b_j,c_j)$.

First we verify whether there exists $m\leqslant M$ with $j = f(m)$ or not.
If such $m$ exists, we substitute $a_j^m b_j^m$ for $c_j$ in $w(a_j,b_j,c_j)$ and verify whether the resulting word is trivial in $\F(a_j,b_j)$ or not.
This can be done effectively.

We claim that in the remaining cases the word $w$ is nontrivial in $G$.
Indeed, if $j\notin \im f$, then $H_j\cong \F(a_j,b_j,c_j)$, hence $w(a_j,b_j,c_j)$ is nontrivial in~$H_j$.
If $j=f(m)$ for some $m>M$, then $w(a_j,b_j,c_j)=w(a_j,b_j,a_j^mb_j^m)$ is nontrivial in $\F(a_j,b_j)$ by Lemma~\ref{two}. \hfill $\Box$

\medskip
\noindent
{\bf Claim 2.} The group $G$ has undecidable $\PP[2]$.

\medskip

{\it Proof.}
The equation $c_k=a_k^xb_k^y$ is solvable if and only if
$k=f(i)$ for some $i$ (in this case $x=y=i$ is the unique solution).
Since the set $\im (f)$ is not recursive, we cannot recognize whether such $i$ exists or not.
Therefore we cannot recognize the existence of such $x$ and $y$.\hfill $\Box$

\bigskip
\noindent
{\bf Claim 3.} The group $G$ has decidable $\PP[1]$.

\medskip

{\it Proof.} Consider an exponential equation
\[
u=v^z,\eqno{(2.3)}
\]
where $u$ and $v$ are nontrivial words in the alphabet $X=\underset{i\in \mathbb{N}}{\cup} X_i$.
In order to decide if it is solvable we may assume that $u\neq 1$ and $v\neq 1$ in $G$.
%Indeed, $G$ is torsion free.

We write $u=u_1u_2\dots u_k$, where $u_i$ is a word in the alphabet $X_{\lambda (i)}$ for some $\lambda(i)$, $i=1,\dots,k$, and $\lambda(j)\neq \lambda(j+1)$ for $j=1,\dots ,k-1$.
Moreover (using decidability of $\WP(G)$), we assume that each $u_i$ represents a nontrivial element of $H_{\lambda (i)}$.
Using conjugation, we may additionally assume that
$\lambda(1)\neq \lambda(k)$ if $k>1$.
Analogously, we write $v=v_1v_2\dots v_{\ell}$.

Suppose that $k>1$. Then the necessary condition for solvability of equation (2.3) is $\ell >1$
and $v_1$ and $v_{\ell}$ belong to different subgroups $H_j$ (determined by $u_1$ and $u_k$).
%Since each $H_i$ is torsion free, we immediately arrive at the situation where $\ell >1$ too,
%and $v_1$ and $v_{\ell}$ belong to different subgroups $H_i$ (determined by $u_1$ and $u_k$).
If this condition is fulfilled, then any possible solution $z$ of equation (2.3) satisfies
$|z|=k/\ell$, and the existence of a solution $z$ can be verified using decidability of $\WP(G)$.

Let $k=1$.
Then the necessary condition for solvability of equation (2.3) is $\ell=1$.
Thus, we assume that $u,v$ are words in the alphabet $X_j$ for some $j$.
We want to solve the equation
\[
u(a_j,b_j,c_j)=v(a_j,b_j,c_j)^z.\eqno{(2.4)}
\]
Without loss of generality, we assume that $u(a_j,b_j,c_j)$ is a reduced cyclic word.
Let $M$ be the maximum of absolute values of exponents of $a_j$ and $b_j$
in $u(a_j,b_j,c_j)$.

First we check whether some $m\in \{1,\dots ,M\}$ satisfies $f(m)=j$.
If such $m$ is found, the equation (2.4) takes the form
\[
u(a_j,b_j,a_j^mb_j^m)=v(a_j,b_j,a_j^mb_j^m)^z ,
\]
and the solvability of this equation can be verified by Lemma~\ref{one}.

If no such $m$ exists, then either $j\notin \im f$, or $j=f(m)$ for some $m>M$.
We claim that in these cases the absolute value of a possible solution $z$ of
equation (2.4) does not exceed the length of the word $u(a_j,b_j,c_j)$ in $\F(a_j ,b_j ,c_j )$.
Indeed, if $j\notin \im f$, then $H_j$ is the free group with basis $\{a_j,b_j,c_j\}$, and the claim follows from Lemma~\ref{one}. If $j\in \im f$, then the claim follows from Lemma~\ref{two}.

Using the estimation for $|z|$ and decidability of $\WP(G)$, we can verify whether equation (2.4) has a solution.\hfill $\Box$ $\Box$

\begin{remark}
One can show that the group $G$ constructed in the proof of Proposition~\ref{four} has solvable conjugacy problem.
However we do not need this for the proof of Theorem~A.
\end{remark}

\section{Proof of Theorem A}
%\section{Proof of the main theorem}

Below we deduce Theorem A from Proposition~\ref{four} and the following result of Ol'shanskii and Sapir.

\begin{theorem} {\rm (see~\cite[Theorem 1]{OS2})}\label{OS2}
Every countable group $G=\langle x_1,x_2,\dots \,|\, R\rangle$ with solvable power problem
is embeddable into a 2-generated finitely presented group
$\overline{G}=\{y_1,y_2\,|\, \overline{R}\}$ with solvable conjugacy and power problems.
\end{theorem}

\begin{remark}\label{rem_computable}
In this remark we recall the main steps of the proof of Theorem \ref{OS2}.
This task is justified by the additional observation that the embedding $\varphi: G\rightarrow \overline{G}$ constructed in the proof of this theorem is {\it computable}.
This means that there exists an algorithm, which given
$i\in \mathbb{N}$, expresses $x_i$ as a word in $y_1$ and $y_2$.
Furthermore, these steps will be also used in arguments of Sections~4 and 5.

\medskip

{\it Four steps in the construction of Ol'shanskii and Sapir.}
Before we start, observe that any countable group $G=\langle x_1,x_2,\dots \,|\, R \rangle$ with solvable power problem has solvable word problem, hence it admits a recursive presentation. Thus, we may assume that the given presentation of $G$ is recursive. Moreover, the solvability of power problem implies the solvability of order
problem (there exists an algorithm which compute orders of elements).

\medskip

{\it Step 1.} In~\cite{Col_1}, Collins noticed that if $H$ is a recursively presented group with solvable power problem and $a,b$ are two elements in $H$ of the same order, then the HNN extension $H_{a,b}=\langle H,t\,|\, t^{-1}at=b\rangle$ has solvable power problem.

Using a sequence of HNN extensions of this type, $G$ can be embedded into a recursively presented group $G_1$ with solvable power problem where every two elements of the same order are conjugate.
Thus the conjugacy problem in $G_1$ is decidable. Moreover, the constructed embedding $\varphi_1:G\rightarrow G_1$ is computable.

\medskip

{\it Step 2.} In~\cite{O1}, Olshanskii suggested the following construction for embedding of countable groups into
2-generated groups.
Let $H=\langle x_1,x_2,\dots \, |\, \mathcal{R}\rangle$ be any countable group. Denote by $\mathcal{R}_1$ the set of words in the alphabet $\{a,b\}$ obtained by substituting the word
$$
A_i=a^{100}b^ia^{101}b^i\dots a^{199}b^i
$$
for every $x_i$ in every word from $\mathcal{R}$. It was shown in~\cite{O1} that the map $x_i\mapsto A_i$, $i\in \mathbb{N}$, extends to an embedding of $H$ into $H_1=\langle a,b\,|\, \mathcal{R}_1\rangle$.
Lemmas 10 and 11 from~\cite{OS2} say that if the group $H$ has decidable word or conjugacy problem or power problem, then the same problem is decidable for the group $H_1$.

Applying this construction, we obtain a computable embedding $\varphi_2:G_1\rightarrow G_2$,
where $G_2=\langle a,b\,|\, R_2\rangle$ is 2-generated, recursively presented, and has solvable power and conjugacy problems.

\medskip

{\it Step 3.} Lemma 12 from~\cite{OS2} says that this $G_2$ can be embedded into a finitely presented group
$G_3=\langle a,b, c_1,\dots ,c_n\,|\, R_3\rangle$ with solvable power and conjugacy problems.
This embedding extends the identity map $a\mapsto a$, $b\mapsto b$.

The corresponding embedding was first described in~\cite{OS1}.
We indicate that $G_2$ and $G_3$ play the roles of $K$ and $H$ in \cite{OS1}.
It is worth mentioning that in \cite{OS1} the set of generators
of $H$ consists of $k$-letters, $a$-letters, $\theta$-letters and $x$-letters.
The subgroup $K$ in $H$ is generated by a subset $\mathcal{A}(P_1 )$ of the set of $a$-letters.
By Lemma~3.9 of \cite{OS1} if
$K = \langle a_1 ,\ldots ,a_m\, |\, \mathcal{R} \rangle$, then
the map $\phi_H$ defined by $a_i \rightarrow a_i (P_1)$ extends to an embedding of $K$ into $H$.
This map is obviously computable.

\medskip

{\it Step 4.} Using the construction from Step 2 once more, we embed $G_3$ into a 2-generated finitely presented group $\overline{G}=\{y_1,y_2\,|\, \overline{R}\}$ with solvable conjugacy and power problems.

Since the embeddings at all steps are computable, their composition $\varphi:G\rightarrow \overline{G}$ is computable
as well.
\end{remark}

%\begin{theorem}\label{ten}
%There exists a finitely presented group with decidable $\PP[1]$ and undecidable $\PP[2]$.
%\end{theorem}

\medskip

{\it Proof of Theorem A.}
By Proposition~\ref{four}, there is a recursively presented group $G$ with decidable $\PP[1]$ and undecidable $\PP[2]$. Using Theorem~\ref{OS2} and Remark~\ref{rem_computable}, we obtain a computable embedding $\varphi:G\rightarrow \overline{G}$,
where $\overline{G}$ is finitely presented and has decidable $\PP[1]$.
%Since $\varphi$ is computable, $\PP[2]$ for $G$
Since $\varphi$ is computable, the undecidability of $\PP[2]$ for $G$ implies
the undecidability of $\PP[2]$ for $\overline{G}$.

Indeed, consider an arbitrary equation $g_0=g_1^xg_2^y$ with $g_0,g_1,g_2\in G$ written as words in the generators of $G$.
Using computability of $\varphi$, we can write $\varphi(g_0)$, $\varphi(g_1)$, $\varphi(g_2)$ as words in the generators of $\overline{G}$.
The equation $\varphi(g_0)=\varphi(g_1)^x\varphi(g_2)^y$ has the same solutions as the original one.
If we could decide whether this equation is solvable, we could decide whether the original equation is solvable.
However $\PP[2]$ is undecidable for $G$.
Hence it is undecidable for $\overline{G}$.\hfill $\Box$

\begin{remark}
The Knapsack counterpart of $\PP [2]$ is also undecidable in $\overline{G}$.
\end{remark}

\section{Estimation functions for solutions of exponential equations}

Let $G$ be a group generated by a set $X$.
For any finite tuple $\bar{g}=(g_0,\dots,g_n)$ of elements of $G$,  the $\infty$-{\it norm} of this tuple is the number
$$
\| \bar{g}\|_X=\max\{|g_0|_X,\dots,|g_n|_X\}.
$$
In the case where $G=\mathbb{Z}$ and $X=\{1\}$ we omit $X$ and write $\|\bar{g}\|$.

\medskip

\begin{definition}~\label{estimation functions} Let $G$ be a group generated by a set $X$.
A function $f:\mathbb{N}\rightarrow \mathbb{N}$ is called
a  $\PP[n]$-{\em bound} for $G$ (with respect to $X$)
%{\it $n$-th estimation function} for solutions of exponential %equations over $G$ with respect to $X$
if for any exponential equation $g_0 = g^{z_1}_{1} \cdot \ldots \cdot g^{z_n}_n$ over $G$ with nonempty set of solutions,
there exists a solution $\bar{k}=(k_1,\dots ,k_n)$ with
$$
\|\bar{k}\|\leqslant f(\|\bar{g}\|_X).
$$
\end{definition}

%For brevity we sometimes write the {\it $n$-th estimation function for $G$ with respect to $X$.

\begin{remark}
Let $G$ be a group and let $X$ and $Y$ be two generating sets of $G$. Suppose that
$$
\underset{y\in Y}{\sup} |y|_X<\infty.
$$
If there exists a (recursive) $\PP[n]$-bound for $G$ with respect to $X$,
then there exists a (recursive) $\PP[n]$-bound for $G$ with respect to $Y$.
\end{remark}

The following lemma relates decidability of $\PP[n]$ in $G$ and existence of a {\it total recursive} $\PP [n]$-bound.
It is a counterpart of the fact that a group $G$ with a finite generating set $X$ has solvable $\WP$ if and only if the Dehn function of $G$ with respect to $X$ is computable.

\begin{lemma} \label{computable}
Let $G$ be a group generated by a finite set $X$.
For any $n\in \mathbb{N}$ the following two conditions are equivalent.
\begin{enumerate}
\item[{\rm (1)}] $\PP[n]$ is decidable in $G$.

\item[{\rm (2)}] $\WP(G)$ is decidable and there exists
a total recursive $\PP [n]$-bound for $G$ with respect to $X$.
%%%%%%%%%%%%%%%%%%%%%%%5
%\item[{\rm (2)}] $\WP(G)$ is decidable and there exists
%a total recursive function
%$f:\mathbb{N}\rightarrow \mathbb{N}$
%such that for any tuple $\bar{g}=(g_0,g_1,\dots,g_n)\in G^{n+1}$
%if the exponential equation
%$g_0 = g^{z_1}_{1} \cdot \ldots \cdot g^{z_n}_n$
%has a solution, then there
%exists a solution $\bar{k}=(k_1,\dots ,k_n)$ with
%$$
%||\bar{k}||\leqslant f(||\bar{g}||_X).
%$$
\end{enumerate}
\end{lemma}

\begin{proof}
$(1)\Rightarrow (2)$.
Suppose that $\PP[n]$ is decidable in $G$.
Then, clearly $\WP(G)$ is decidable.
Now we define the desired function
$f:\mathbb{N}\rightarrow \mathbb{N}$ at arbitrary point $m\in \mathbb{N}$
in four steps.

\begin{enumerate}
\item[1)] Let $B(m)$ be the set of all tuples $\bar{g}=(g_0,g_1,\dots,g_n)$ of words in the alphabet $X$ satisfying
$\|\bar{g}\|_X\leqslant m$.
Since $X$ is finite, the set $B(m)$ is finite and we can compute it.

\item[2)] Let $B(m)'$ be the subset of $B(m)$ consisting of the tuples $\bar{g}=(g_0,g_1,\dots,g_n)$ for which that the equation $g_0 = g^{x_1}_{1} \cdot \ldots \cdot g^{x_n}_n$ has a solution.
We can compute $B(m)'$ using decidability of $\PP[n]$ in $G$.

\item[3)] For each tuple $\bar{g}\in B(m)'$ we can find some solution $\bar{k}=(k_1,\dots,k_n)$ of the equation $g_0 = g^{z_1}_{1} \cdot \ldots \cdot g^{z_n}_n$ using
an effective enumeration of $n$-tuples of integers
and the decidability of $\WP(G)$.
We denote this solution by $\bar{k}(\bar{g})$.
%and the maximum of absolute values of its components by $\|\bar{k}(\bar{g})\|$.

\item[4)] Finally we set $f(m)$ to be the maximum of
$\|\bar{k}(\bar{g})\|$ over all
$g\in B(m)'$ if $B(m)'\neq \emptyset$, and we set $f(m)=1$
if $B(m)'=\emptyset$.
\end{enumerate}

\noindent
The function $f$ is total and recursive, and satisfies Definition~\ref{estimation functions}.
%and satisfies the condition stated in~(2).

\medskip

$(2)\Rightarrow (1)$.
Consider an exponential equation
$g_0 = g^{z_1}_{1} \cdot \ldots \cdot g^{z_n}_n$ over $G$.
%%%%%%%%%%%
%Condition (2) gives us the estimation
%$\max |k_i|\leqslant f(|\bar{g}|_X)$ for a possible solution
%$(k_1,\dots k_n)$.
%Now we go through all such $n$-tuples
To decide, whether this equation has a solution, we verify whether
the equality $g_0 = g^{k_1}_{1} \cdot \ldots \cdot g^{k_n}_n$ holds for at least one tuple $\bar{k}=(k_1,\dots,k_n)\in \mathbb{Z}^n$ with $\|\bar{k}\|\leqslant f(\|\bar{g}\|_X)$.
The verification for a concrete tuple $\bar{k}$ can be done using $\WP(G)$.
 \end{proof}

\begin{remark} In~\cite{Kharlampovich}, Kharlampovich constructed a group $G$ which is finitely presented in the variety $x^m=1$
and has undecidable word problem.
By Lemma~\ref{computable}, $\PP[n]$ is undecidable for each $n$.
On the other hand the constant function $f(k) = m$, $k\in \mathbb{N}$,
is a total recursive $\PP [n]$-bound for $G$.
%satisfies the condition in statement (2) of this lemma, i.e.
\end{remark}

%The following proposition shows that in Lemma \ref{computable}
%one cannot demand that the first estimation function is primitively recursive.

The following proposition shows that there is a finitely presented group
with decidable $\PP[1]$ which does not have a primitively recursive $\PP[1]$-bound.
%in Lemma \ref{computable}
%the decidability of $\PP[n]$ does not imply the existence of a primitively recursive $\PP[n]$-bound.

\begin{proposition}\label{unsolvable}
There exists a finitely presented group $G = \langle X |\, \mathcal{R}\rangle$ with decidable $\PP[1]$
and there exists a collection of elements $(c_n)_{n\in \mathbb{N}}$ of $G$ such that
%$|c_n |=\infty$ for $n\in \mathbb{N}$ and
the following holds.
\begin{enumerate}
\item[\rm (1)] For any $n$ the equation $c_1=c_n^x$ has a unique solution, say $k_n$; this solution is positive.

\item[\rm (2)] There is no primitively recursive function $f$ such that $k_n\leqslant f(\max\{|c_n|_X,|c_1|_X\})$.
\end{enumerate}

%The group $G\ast F_2$ satisfies these properties and is relatively hyperbolic with respect to $G\cup F_2$.
\end{proposition}

\begin{proof}
We enumerate all primitively recursive functions $g_1,g_2,\dots $ and, for any~$n\in~\mathbb{N}$, we define
a function $f_n:\mathbb{N}\rightarrow \mathbb{N}$ by the rule
$$
f_n(x)=\overset{n}{\underset{i=1}{\sum}} \overset{x}{\underset{j=1}{\sum}}\, \, g_i(j),\hspace*{2mm} x\in \mathbb{N}.
$$
%$g_i'(x)=\overset{x}{\underset{j=1}{\sum}} g_i(j)$ and $f_i=g_1'+\dots +g_i'$.
Clearly, $f_n$ is primitively recursive, nondecreasing, $g_n\leqslant f_n$, and $f_n\leqslant f_{n+1}$.
Finally, we define a function $F:\mathbb{N}\rightarrow \mathbb{N}$ by the rule
$$
F(n)=n! (f_n(100 n+14500)+1).
$$
Clearly, $F$ is recursive.
We also define rational numbers $c_1=1$ and $c_n=\frac{1}{F(n)}$ for $n\geqslant 2$.
Then the following recursive presentation (written multiplicatively) defines the group $\mathbb{Q}$:
%Slightly correcting $F$ if necessary we may arrange that the following recursive presentation (written multiplicatively) defines $\mathbb{Q}$:
$$
\langle c_1,c_2,\dots \,|\, c_n^{F(n)}=c_1,\, [c_n,c_m]=1\, (n,m\in \mathbb{N})\rangle.
$$
It is obvious that $\PP[1]$ is decidable for this presentation.
We embed $\mathbb{Q}$ into a finitely presented group $G_3$ by the Ol'shanskii -- Sapir construction which we described in
Steps 1-4 in Section 3. Note that since $\mathbb{Q}$ has decidable conjugacy problem, we do not need to do Step 1.
Thus, we start with Step 2, where we use the following map.
\begin{itemize}
\item
Let $\varphi_2$ maps each $c_i$ to the word
$a^{100}b^i a^{101}b^i\ldots  a^{199}b^i$ (of length $100i + 14500$), $i \in \mathbb{N}$.
\end{itemize}

By Step 2, $\varphi_2$ extends to an embedding $\varphi_2:\mathbb{Q}\rightarrow G_2$,
where the group $G_2=\langle a,b\,|\, R_2\rangle$ is 2-generated, recursively presented, and has solvable power and conjugacy problems.
Then we only apply Step 3. By this step the map $a\mapsto a$, $b\mapsto b$ extends
to an embedding $\varphi_3:G_2\rightarrow G_3$, where $G_3=\langle X|\, R_3\rangle$
is a finite presentation with solvable power and conjugacy problems, and $\{a,b\}\subseteq X$.
We set $G=G_3$.
%We do not want the group to be two-generated.
%Let $X$ be a finite generating set of $G_1$.

The statement (1) is valid: for any $n$, the equation $c_n^x=c_1$ has a unique solution, namely $k_n=F(n)$.
To prove statement (2), we first observe that
$$
\max\{|c_n|_X,|c_1|_X\}\leqslant  \max\{|c_n|_{\{a,b\}},|c_1|_{\{a,b\}}\}{\leqslant} 100n+14500.\eqno{(4.1)}
$$
Suppose that statement (2) is not valid, i.e. there exists a primitively recursive function $g_m$ such that
$$
k_n\leqslant g_m(\max\{|c_n|_X,|c_1|_X\})\eqno{(4.2)}
$$ for any $n$.
Using that $g_m\leqslant f_m$ and that $f_m$ is nondecreasing, we deduce from (4.1) and (4.2) that
$$
F(n)\leqslant f_m(100n+14500)
$$ for any $n$.
In particular, $F(m)\leqslant f_m(100m+14500)$.
This contradicts the definition of $F$.
\end{proof}

\begin{remark}\label{complexity_remark_1}
Theorem 4.2 in \cite{LZ_2} states that the Knapsack Problem in the Baumslag -- Solitar group ${\rm BS}(1,2)$
is NP-complete but the
%$\PP[3]$-bound
function estimating nonnegative solutions of exponential equations with 3 variables
%the $\PP[n]$-bounds can be chosen to be smaller than double exponential.
%of solutions of the corresponding exponential equations
cannot be essentially smaller than the exponential function
(note that in~\cite{LZ_2}, the authors use the binary representation of natural numbers).
This statement can be considered as a counterpart of Proposition~\ref{unsolvable}
at the level of polynomial computability.
\end{remark}

\begin{remark}\label{complexity_remark_2}
However for hyperbolic groups
such functions for any $n$ can be chosen to be linear (a polynomial estimation was known earlier, see~\cite{MNU}). This follows from a forthcoming paper~\cite{Bogo_Bier} of the first named author and A. Bier. It is proved in~\cite{Bogo_Bier} that similar linearity result holds for acylindrically hyperbolic groups in the case of loxodromic coefficients. Another result in~\cite{Bogo_Bier} states that there is a linear reduction to peripheral subgroups in the case of relatively hyperbolic groups.

\end{remark}

\section{Restricted versions of $\PP [n]$}

We introduce two new algorithmic problems which can be considered as fragments of $\PP[n]$.
Informally we call them the the {\it left and the right fragments} of~$\PP[n]$.
We show that these fragments can take diverse computational complexities for the same finitely presented group, see Corollary~\ref{cor1/2}.

\subsection{Definitions and observations}

Below we assume that $G$ is given by a recursive presentation and $X$ is the corresponding set of generators.

\begin{definition}\label{fragments}
(1) Let $g_1 , \ldots , g_n \in G$.
By $\PP [G, g_1 , \ldots , g_n ]$ we denote the set of all $g \in G$
such that the equation
$g = g^{z_1}_{1} \cdot \ldots \cdot g^{z_n}_n$ has a solution which is a tuple of integers. \\
(2) For a fixed $g\in G$ let $\PP [g,G^n]$
consist of all tuples $\bar{g} = (g_1 ,\ldots ,g_n )$
such that the equation
$g = g^{z_1}_{1} \cdot \ldots \cdot g^{z_n}_n$ has a solution which is a tuple of integers.
\end{definition}
Note that for a tuple of units $\bar{e}$ the membership problem
for $\PP [G, \bar{e}]$ is equivalent to the word problem.
Decidability of the problem $\PP[n]$ is a uniform form
of decidability of all $\PP [G, \bar{g}]$ (resp. $\PP [g,G^n]$).
Indeed,  if for each $g\in G$ there is an algorithm (depending on the word $g$) which decides the membership problem for $\PP [g,G^n]$, then $\PP[n]$ is decidable.
The similar statement holds for problems
$\PP [G, g_1 , \ldots , g_n ]$.

\begin{remark} Suppose that $G$ is a group given by a recursive presentation.
Let $g$ be an nontrivial element of $G$.
Suppose that $\PP [G, g]$ is decidable in $G$ and the order of $g$ is known.
Then $\WP$ is decidable in $G$.

Indeed, in order to determine whether a given $h$ is trivial in $G$, we first verify whether
$h$ is a power of $g$.
%the power problem for the pair $h,g$.
If $h$ is not a power of $g$ then $h\not= 1$.
If we know that $h$ is a power of $g$ let us start a diagonal computation for verification of the following equalities:
$h=1$, $h=g$, $\ldots, h=g^k, \ldots$ .
Here we use the recursive presentation of~$G$.
At some stage we will find a number $k$ with $h = g^k$. Since the order of $g$ is known, we can check, whether $h=1$ or not.
\end{remark}

\begin{quote}
{\em What is the relationship among  possible algorithmic complexities
of $\PP[n]$, $\PP [g,G^n]$ and $\PP [G, g_1 , \ldots , g_n ]$ for tuples
$g, g_1 , \ldots , g_n \in G$?  }
\end{quote}

\medskip

Some additional issue of this problem is nicely illustrated by  McCool's example from \cite{MC}.
Having a computable function $F$ with a non-computable $\im F$
the group
\[
\big\langle \bigcup_{n\in \mathbb{N}} \{ a_n, b_n \} \,\,|\, \{ [a_n, b_n] = 1 \, | \, n \in \mathbb{N} \} , \,\{  a_{F(n)}=b_{F(n)}^n \, | \, n\in \mathbb{N} \} \big\rangle
\]
has decidable $\WP$ but undecidable $\PP [1]$.
On the other hand each $\PP [g_0 ,G]$ (resp. $\PP [G, g_0 ]$) is a decidable problem whenever the set
\[
\{ n \in \im\, F \, | \, a_n \mbox{ or } b_n \mbox{ occurs in } g_0\}
\,
\]
is provided.
This claim can be easily verified (for example see arguments
in the proof of Theorem \ref{1/2} below).

\subsection{Example}

Let $p_n$ denote the $n$-th prime number. For any function $F:\mathbb{N}\rightarrow \mathbb{N}^2$ and any $n\in \mathbb{N}$ we denote
$(\im F)_n = \{ m\in \mathbb{N}\, |\, (n ,m)\in \im (F)\}$ and write $F=(F_1,F_2)$.

\medskip

Let $F:\mathbb{N}\rightarrow \mathbb{N}^2$ be a total, injective, computable function such that for any $n\in \mathbb{N}$ we have either $(\im F)_n = \emptyset$  or
$$
p_n\in (\im F)_n\subseteq\{ p^k_{n}\, |\, k\in \mathbb{N} \setminus \{ 0 \}\}.
$$
Thus all sets $(\im F)_n$, $n\in \mathbb{N}$, are pairwise disjoint.
We put
\[
X= \{ a_n\, |\, n\in \mathbb{N}\} \cup \{ b_m\, |\, m \mbox{ is a power of a prime number}, m\neq 1\}
\]
and consider the group with the following recursive presentation

\[
G=\big\langle X\,\,|\, \{ [a_n, b_m]=1 \, | \, \exists\, k \, (m = p^k_n) \}
, \, a_{F_1(n)}=b_{F_2(n)}^n \, | \, n\in \mathbb{N}\big\rangle.\eqno{(5.1)}
\]

\begin{theorem} \label{1/2}
For the above defined group $G$ the following statements are valid.
\begin{enumerate}
\item[{\rm (1)}] $\CP(G)$ is decidable. In particular, $\PP[e,G]$ and $\PP[G,e]$ are decidable.\vspace*{1mm}

\item[{\rm (2)}] $\PP[1]$ is undecidable for $G$ if the set $\im\, (F_1 )$ is not computable.\vspace*{1mm}

\item[{\rm (3)}] For any fixed $g_0\in G$ the problem $\PP [g_0 ,G]$
(resp. $\PP [G,g_0 ]$) is decidable or there is a number $n$
such that $\PP [g_0 ,G]$ (resp. $\PP [G,g_0 ]$) is Turing reducible to $(\im\, F)_n$.
Each of these possibilities can be effectively recognized and the corresponding number $n$ can be computed.\vspace*{1mm}

\item[{\rm (4)}] If $n\in \im\, F_1$ then the problem $\PP[a_{n},G ]$ is computably equivalent to the membership problem for $(\im F)_{n}$.
\end{enumerate}
\end{theorem}

\begin{proof}

Before we start to prove these statements, we establish the structure of $G$.
We decompose $X=\underset{i\in \mathbb{N}}\cup X_i$, where
$$
X_i=\{a_i\} \cup \{ b_j \, | \, j \mbox{ is a power of } p_i \}.
$$
Let $H_i$ be the subgroup of $G$ generated by $X_i$. Then
\[
G=\underset{i\in \mathbb{N}}{\ast} H_i. \eqno{(5.2)}
\]

\noindent
To describe the structure of $H_i$, we first introduce the following subgroups of $H_i$:
$$
H_i^{-}=\langle b_j\,|\, j \mbox{ is a power of } p_i \mbox{ satisfying } j\notin (\im F)_i\rangle,
$$
$$
H_i^{+}=\langle b_j\,|\, j \mbox{ is a power of } p_i \mbox{ satisfying } j\in (\im F)_i\rangle.
$$
Then $H_i^{-}$ is the free product of all its subgroups $\langle b_j\rangle$, and $H_i^{+}$ is the amalgamated product over $\langle a_i\rangle$ of all its subgroups $\langle b_j\rangle$.
Moreover, we have
\[
H_i=(\langle a_i\rangle\times H_i^{-})\underset{\langle a_i\rangle}{\ast} H_i^{+}.\eqno{(5.3)}
\]
Note that $\langle a_i\rangle$ is the center of $H_i$.

\medskip

Before we start the proof of statement (1), we make the following important observation.

{\bf Observation.} Let $i,j\in \mathbb{N}$ and $k\in \mathbb{Z}$. Then $b_j^{k}$ is a power of $a_i$ if and only if $j$ is a power of $p_i$ and there exists a positive divisor $d$ of $k$ such that $F(d)=(i,j)$.\break
We can recognize the existence of such $d$ since $F$ is computable. If such $d$ exists, then $a_i=b_j^d$ and hence $a_i^{k/d}=b_j^k$.

\medskip

{\it Proof of statement} (1). First we prove that $\WP(G)$ is decidable.
Using the normal form of an element of the free product (5.2), we reduce this problem to the following one.
Given $i\in \mathbb{N}$ and given a cyclically reduced nonempty
word $a^s_i w(\bar{b})$, where $w(\bar{b})$ is over
$X_i \setminus \{ a_i \}$, decide whether the corresponding element of $H_i$ is trivial or not.

We may assume that the word $w(\bar{b})$ is nonempty.
Indeed, otherwise $a^s_i w(\bar{b})$ lies in the cyclic subgroup
$\langle a_i \rangle$ of $H_i$ and therefore is trivial exactly when $s=0$.

Using the above observation, we verify whether some subword $b_j^k$ of $w(\bar{b})$ is a power of $a_i$ or not.
If no one such subword is a power of $a_i$, then the element $a^s_i w(\bar{b})$ is nontrivial in the amalgamated product (5.3).
Suppose that some subword $b_j^k$ of $w(\bar{b})$ is a power of $a_i$, say $a_i^{\ell}=b_j^k$.
Since $a_i$ lies in the center of $H_i$, we can move this subword to the left and adjoin to $a^s$.
After this operation $|w(\bar{b})|_{X_i \setminus \{ a_i \}}$ decreases and we can proceed by induction.

\medskip

Now we show that the conjugacy problem in $G$ is decidable.
Using (5.2), we reduce this problem to the conjugacy problem in the groups $H_i$, $i\in \mathbb{N}$. By (5.3), each $H_i$ is an amalgamated product over the center of $H_i$.
This fact, the decidability of $\WP(G)$, and a criterium for conjugacy of elements in amalgamated products (see~\cite[Chapter IV, Theorem 2.8]{LS}), imply that there is a universal algorithm deciding the conjugacy problem in each $H_i$ and hence in $G$.

\medskip

{\it Proof of statement} (2). This statement easily follows from the the equivalence
\[
n\in \im F_1\,\, \Longleftrightarrow\,\,\, a_n \mbox{ is a power of  }b_{p_n}.
\]
Indeed, if $\im F_1$ is not computable, we cannot decide,
given $n\in \mathbb{N}$, whether
the equation $a_n=b_{p_n}^x$ has a solution or not.

\medskip

{\it Proof of statement} (3). For a fixed element $g_0\in G$ we study the problem $\PP [g_0 ,G]$.
Given another element $g_1\in G$, we shall consider the exponential equation
% $g_0=g_1^z$.
\[
g_0=g_1^z.
\]
We may assume that $g_0\neq 1$, otherwise $\PP[g_0,G]$ is decidable since $G$ is torsion-free and $\WP(G)$ is decidable. Having $g_0\neq 1$, we may assume that $g_1\neq 1$. Standardly, we assume that $g_0$ and $g_1$ are represented by some words $u$ and $v$ in the alphabet $X=\underset{i\in \mathbb{N}}{\cup} X_i$.
Thus, we consider the following exponential equation in $G$:
$$
u=v^z.\eqno{(5.4)}
$$
We write $u=u_1u_2\dots u_k$, where $u_i$ is a word in the alphabet $X_{\lambda (i)}$ for some $\lambda(i)$, $i=1,\dots,k$,
and $\lambda(j)\neq \lambda(j+1)$ for $j=1,\dots ,k-1$.
Moreover, we assume that each $u_i$ represents a nontrivial element of $H_{\lambda (i)}$.
This can be recognized by decidability of $\WP(G)$.
Using conjugation, we may additionally assume that
$\lambda(1)\neq \lambda(k)$ if $k>1$.
Analogously, we write $v=v_1v_2\dots v_{\ell}$.

Suppose that $k>1$. Then the necessary condition for solvability of equation (5.4) is $\ell >1$
and $v_1$ and $v_{\ell}$ belong to different subgroups $H_i$ (determined by $u_1$ and $u_k$).
If this condition is fulfilled, then any possible solution $z$ of equation (5.4) satisfies
$|z|=k/\ell$, and the existence of a solution $z$ can be verified using decidability of $\WP(G)$.

Note that until this moment the corresponding algorithm is uniform on $u$ and~$v$.
In the following case we will call some oracle depending on $u$.

Suppose that $k=1$.
Then $u\in H_n$ for $n=\lambda(1)$, and this $n$ can be determined using the definition of $X_n$.
Using the procedure described in the proof of $\WP(G)$,
we write $u$ in the normal form with respect to the amalgamated product (5.3), i.e. we write
$u = a^s_n b^{s_1}_{i_1} \ldots b^{s_p}_{i_p}$
where, in particular, each $b^{s_1}_{i_1}, \ldots , b^{s_p}_{i_p}$
does not have a subword which is a power of $a_n$. Conjugating, we may additionally assume that $u$ is cyclically reduced in the sense that $i_1\neq i_p$ if $p>1$.

Furthermore, we may now assume that $v$ belongs to $H_{n}$ too. We also
write $v$ in the normal form with respect to the amalgamated product (5.3),
$v = a^t_n b^{t_1}_{j_1} \ldots b^{t_q}_{j_q}$

If $p >1$, then the necessary condition for solvability of equation (5.4) is $q>1$. In this case any possible solution $z$ of (5.4) satisfies $|z|=p/q$, and the existence of the corresponding $z$ can be verified
using decidability of $\WP(G)$.

Suppose that $p=1$. Then the necessary condition for solvability of (5.4) is $q =1$ and $i_1 = j_1$.
In this case we only need to verify the existence of $z$ satisfying (5.4) in the  group $\langle a_n , b_{i_1} \rangle$.
This is the only place where we need the oracle for $(\im F)_n$.
Verifying whether $p_{n}\in (\im F)_{n}$, we decide if $n \in \im F_1$.
If this happens, we easily compute in the oracle $(\im F)_n$ the relation from the presentation (5.1) of the form $a_{n} = b^r_{i_1}$ if it exists, and if it does not exist we recognize this.
In the latter case any possible solution $z$ of equation (5.4) satisfies $|z|\leqslant |u|_{X_n}$.
In the former case
$\langle a_n , b_{i_1} \rangle = \langle b_{i_1} \rangle$.
Substituting the appropriate power of $b_{i_1}$ instead of $a_n$ both in $u$ and $v$ we obtain an equation in the cyclic group
$\langle b_{i_1} \rangle$.
Now the number $z$ can be computed.
This gives an appropriate algorithm which is computable with respect to $(\im\, F)_n$.

The case $p=0$ is trivial and we leave it to the reader.

This completes the proof of statement~(3) for $\PP [g_0 ,G]$. The argument for $\PP [G,v]$ is analogous.

\medskip

{\it Proof of statement} (4). Let $n\in \im\, F_1$.
The following equivalence recognizes $j \in (\im\, F)_n$ under the oracle for $\PP [a_n, G]$:
\[
j\in (\im\, F)_n\,\, \Longleftrightarrow\,\, j \mbox{ is of the form }p^k_n
\mbox{ and } a_n \mbox{ is a power of  }b_j .
\]
\end{proof}

\begin{remark}
Statements (1) and (4) also hold for the corresponding versions of the Knapsack Problem.
\end{remark}

The following corollary is Theorem B from Introduction.

\begin{corollary} \label{cor1/2}
There exists a finitely presented torsion-free group $G$ with decidable conjugacy problem and undecidable $\PP[1]$ such that any r.e. Turing degree is realised as the Turing degree of the problem
$\PP[g,G]$ for appropriate $g\in G$.
\end{corollary}

\medskip

{\it Proof.} First we construct a recursively presented group $G$ with these properties.
Let $\varphi (x,y)$ be Kleene's universal computable function.
Let
\[
W= \{ (x,z) \, | \, \exists y \,( z=\varphi (x,y)) \}
\]
and  let $\Phi: \mathbb{N}\rightarrow \mathbb{N}^2$ be a total, injective, computable function with
$\im\, \Phi = W$. Below we use notations introduced at the beginning of this subsection.
Obviously, the sets $(\im\, \Phi)_n$, $n\in \mathbb{N}$, have all possible r.e. Turing degrees.
Now we extend the set $W$ as follows:
\[
\widehat{W}= W\cup \{ (x,1) \, | \, \exists z :(x,z)\in W \}.
\]
Let $\widehat{\Phi}: \mathbb{N}\rightarrow \mathbb{N}^2$ be a total, injective, computable function
with $\im\, \widehat{\Phi} = \widehat{W}$. We have $(\im\, \widehat{\Phi})_n=(\im\, \Phi)_n\cup \{1\}$
for any $n\in \mathbb{N}$.
Therefore the sets $(\im\, \widehat{\Phi})_n$, $n\in \mathbb{N}$ have all possible r.e. Turing degrees as well.

Now we define a function $F:\mathbb{N}\rightarrow \mathbb{N}^2$  by the formula
$$
F=f\circ \widehat{\Phi},
$$
where $f:\mathbb{N}^2\rightarrow \mathbb{N}^2$ is the function
sending each $(n,m)$ to $(n, p^m_n)$.
The function $F$ satisfies
the conditions formulated at the beginning of this subsection, since it is
total, injective, computable, and for any $n\in \mathbb{N}$ we have
$$
p_n\in (\im F)_n\subseteq\{ p^k_{n}\, |\, k\in \mathbb{N} \setminus \{ 0 \}\}.
$$
Finally, we define a recursively presented group $G$ by formula (5.1) and apply Theorem~\ref{1/2}.
By statements~(1) and (2) of this theorem, $\CP(G)$ is decidable and $\PP[1]$ is undecidable for $G$.

The statement of the corollary about Turing degrees follows from statement (4) of Theorem~\ref{1/2} which says that,
for any $n\in \mathbb{N}$, the problem $\PP[a_{n},G ]$ is computably equivalent to
the membership problem for $(\im\, F)_{n}$.
It remains to note that
$$
(\im\, F)_n=\{p_n^m\,|\, m\in (\im\, \widehat{\Phi})_n\},
$$
therefore these sets have all possible r.e. Turing degrees.

\medskip

$\bullet$ Now we embed the group $G$ into a finitely presented group $\overline{G}$
using Ol'shanskii~-- Sapir construction explained in Remark~\ref{rem_computable}.
Note that we do not need to do Step~1 there since $G$ already has solvable conjugacy problem, and we do not need
to do Step 4 since we do not specially want $\overline{G}$ to be 2-generated.

Thus, using notations of this remark, we may assume that $G=G_1$ and that we have embeddings $G_1\rightarrow G_2\rightarrow G_3$, where $G_3=\overline{G}$.
Simplifying notation, we assume $G_1\leqslant G_2\leqslant G_3$. By this construction, $G_3$
has solvable conjugacy problem if $G_1$ has solvable conjugacy problem. The latter is valid, hence $\overline{G}$ has solvable conjugacy problem. It remains to prove the following claim.

%%Denote $G_1=G$ and $G_4=\overline{G}$.
%According to this remark we have a sequence of embeddings $G_1\rightarrow G_2\rightarrow G_3\rightarrow G_4$,
%where $G_1=G$ and $G_4=\overline{G}$.

%For that we cannot use Theorem~\ref{OS2}, since $\PP[1]$ is undecidable for $G$. However, we can use this construction

\medskip

{\it Claim.} For any $g\in G$ the problems $\PP[g,G_1]$ and $\PP[g,G_3]$ are computationally equivalent.

\medskip

{\it Proof.}
The computational equivalence of $\PP[g,G_1]$ and $\PP[g,G_2]$ follows from the proof of Lemma 11 in~\cite{OS2}.
(We stress that we use the proof and not the formulation of this lemma which requires solvability of power problem in $G_1$.) Indeed, given an exponential equation $g=u^z$ with $u\in G_2$,
the proof (depending on~$u$) either recursively reduces this equation to $\PP[g,G_1]$ or gives a linear upper bound for $|z|$ in terms of $g$ and $u$.

The computational equivalence of $\PP[g,G_2]$ and $\PP[g,G_3]$ analogously follows from the proof of Lemma 12 in~\cite{OS2}.
\hfill $\Box$ $\Box$

\section{Further observations. Complexity}

Applying the approach of \cite{BCR} we obtain the following proposition.

\begin{proposition}
\begin{enumerate}
\item[{\rm (1)}] Detecting a group with decidable $\PP [n]$ is $\Sigma^0_3$ in the class of recursively presented groups.
\item[{\rm (2)}] The same conclusion holds both for the problem $\underset{n\in \mathbb{N}}\bigcup \PP[n]$ and the Knapsack Problem.
\end{enumerate}
\end{proposition}

\begin{proof}
We will use standard terminology from \cite{soare}.
The universal computable function $\varphi (x,y)$ will be applied to several families of objects.
As usual these objects are coded by natural numbers.

Take a computable indexation
$G_i = \langle X \, | \, \mathcal{R}_i \rangle$, $i\in \omega$,
of all recursively presented groups with respect
to generators $X= \{ x_1, x_2 ,\ldots \}$.
Fix an algorithm which for the input $(i,s)$ outputs the $s$-th
equality of the form $w = 1$ satisfied in $G_i$.
We see that the set of pairs $(G_i ,w)$, where $G_i \models w=1$  with $w \in \F(X)$ is computably enumerable.
There also exists a computable enumeration of the set of pair
$(G_i ,\bar{w})$ where $\bar{w} = (w_0 , w_1 , \ldots , w_n )$ belongs to $\PP [n](G_i )$.
Thus the set
$I_{\bar{w}} = \{ G_i \, | \, (w_0 , w_1 , \ldots , w_n )\in \PP[n](G_i)\}$ is computably enumerable.
These sets belong to $\Sigma^0_1$.
On the other hand the set $\bar{I}_{\bar{w}} = \{ G_i \, | \, (w_0 , w_1 , \ldots , w_n )\not\in \PP[n](G_i)\}$ belongs to $\Pi^0_1$.
The property $(w_0 , w_1 , \ldots , w_n )\not\in \PP[n](G_i)$ exactly means that for any
$(s_0 , s_1 ,\ldots ,s_n )$ the equality
$w_0 = w^{s_1}_1 \cdot \ldots \cdot w^{s_n}_n$ is not recognized in $G_i$ at step $|s_0|$.
Developing these observations we formulate decidability of $\PP [n]$ for $G_i$ as follows.
\begin{quote}
There is a number $m\in \mathbb{N}$ such that for any tuple
$w_0 , w_1 , \ldots , w_n\in \F(X)$
and any $(s_0 , s_1 ,\ldots ,s_n )\in \mathbb{Z}^{n+1}$
there exist numbers $s, t\in \mathbb{N}$ such that the following properties hold:
\begin{itemize}
\item the algorithm $\varphi (m, . )$ applied to the code of $\bar{w}$ gives the value $0$ or $1$ at step $s$;
\item the algorithm $\varphi (m, . )$ applied to the code of $\bar{w}$ gives the value $0$ at step $s$ or the membership
%when the value is $1$ the algorithm $\varphi (m,.)$ confirms
$G_i \in I_{\bar{w}}$ is confirmed at step $t$ of computation;
\item  the algorithm $\varphi (m, . )$ applied to the code of $\bar{w}$ gives the value $1$ at step $s$
or the equality $w_0 = w^{s_1}_1 \cdot \ldots \cdot w^{s_n}_n$ is not recognized at step $|s_0|$.
\end{itemize}
\end{quote}
The second statement of the proposition is similar.
\end{proof}

{\bf Question.}
Are $\PP [n]$ and the Knapsack Problem $\Sigma^0_3$-complete in the class of recursively presented groups?

\def\refname{REFERENCES}


\bigskip
\begin{thebibliography}{99}

\bibitem{Bogo_Bier} A. Bier, O. Bogopolski, {\it Exponential equations in acylindrically hyperbolic groups}, 2021. In progress.

\bibitem{BCR} I. Bilanovic, J. Chubb and S. Roven, {\it Detecting properties from presentations of groups}, Arch. Math. Log.
{\bf 59} (2020), 293 - 312.

\bibitem{Col_1} Donald J. Collins, {\it On embedding groups and the conjugacy problem}, J. London Math. Soc.,
(2) 1 (1996), 674-682.

\bibitem{Dudkin} F. Dudkin, A. Treyer, {\it Knapsack problem for Baumslag-Solitar groups}, Siberian Journal
of Pure and Applied Mathematics, {\bf 18} (4) (2018), 43-55.

%\bibitem{E} Yu. L. Ershov, {\it Numbering Theory}, Nauka, 1977.

\bibitem{Frenkel} E. Frenkel, A. Nikolaev, A. Ushakov, {\it Knapsack problems in products of groups}, Journal of Symbolic Computation, {\bf 76} (2016), 96-108.

\bibitem{GKLZ} M. Ganardi, D. K{\"o}nig, M. Lohrey, G. Zetzsche, {\it Knapsack problems for wreath products}.
In Proceedings of STACS 2018, vol. {\bf 96} of LIPIcs, 1-13.
%Available at, https://arxiv.org/pdf/1709.09598.pdf

\bibitem{Kharlampovich} O. Kharlampovich, {\it The word problem for the Burnside varieties}, Journal of Algebra,
{\bf 173}, no.~3 (1995), 613-621.

\bibitem{KLZ} D. K{\"o}nig, M. Lohrey, G. Zetzsche, {\it Knapsack and subset sum problems for nilpotent,
polycyclic, and co-context-free groups}, In Algebra and Computer Science, volume {\bf 677} of Contemporary Mathematics, pages 138-153. American Math. Society, 2016.
%https://arxiv.org/pdf/1507.05145.pdf

\bibitem{Lohrey_1} M. Lohrey, {\it Rational subsets of unitriangular groups}, Int. J. Algebra Comput.,
{\bf 25}, (1-2) (2015), 113-121.

\bibitem{Lohrey_2} M. Lohrey, {\it Knapsack in hyperbolic groups}, J. of Algebra,
 vol. {\bf 545} (1) (2020), 390-415.
%Available in Arxiv: https://arxiv.org/pdf/1807.06774.pdf

\bibitem{LZ_1} M. Lohrey, G. Zetzsche, {\it Knapsack in graph groups, HNN-extensions and amalgamated products},
Theory of Computing Systems, {\bf 62} (1) (2018), 192-246.
%https://arxiv.org/pdf/1509.05957.pdf

\bibitem{LZ_2} M. Lohrey, G. Zetzsche, {\it Knapsack and the power word problem in solvable Baumslag-Solitar groups},
2020.
https://arxiv.org/pdf/2002.03837.pdf

\bibitem{LS} R.C. Lyndon, P.E. Schupp, {\it Combinatorial group theory}, Springer, Berlin, 1977.

%\bibitem{AIM} A. I. Maltsev, {\it Algorithms and recursive %functions}, Nauka, Moscow, 1965.

\bibitem{MC} J. McCool, {\it Unsolvable problems in groups with solvable word problem}, Can. J. Math., XXII (1970), no. 4, 836-838.


\bibitem{MT} A. Mishchenko, A. Treier, {\it Knapsack problem for nilpotent groups}, Groups, Complexity and Cryptology,
{\bf 9} (1) (2017), 87-98.
%https://arxiv.org/pdf/1606.08584.pdf

\bibitem{MNU} A. Myasnikov, A. Nikolaev, A. Ushakov, {\it Knapsack problems in groups},
Mathematics of Computations, {\bf 84} (292) (2015), 987-1016.

\bibitem{O1} A.Yu. Olshanskii, {\it SQ-universality of hyperbolic groups}, Mat. Sb., {\bf 186}, no. 8 (1995), 119-132.

\bibitem{OS1} A.Yu. Olshanskii and M.V. Sapir, {\it The conjugacy problem and Higman embeddings},
Memoirs of the AMS 170 (2004), no. 804 p.p. vii+131 .

\bibitem{OS2} A.Yu. Olshanskii and M.V. Sapir, {\it Subgroups of finitely presented groups with solvable conjugacy
problem}, Int. J. Algebra and  Comput., {\bf 15}, no. 5-6 (2005), 1-10.

\bibitem{OS3}  A.Yu. Olshanskii and M.V. Sapir, {\it Algorithmic problems in groups with quadratic Dehn function}, 
2020. https://arxiv.org/pdf/2012.10417.pdf

\bibitem{soare} R. I. Soare, Turing Computability.
Theory and Applications, Springer, Berlin, 2016.

\end{thebibliography}
\end{document}